\documentclass[10pt]{article}
\font\smallit=cmti10

\usepackage{caption}
\usepackage{color}
\usepackage{amssymb,latexsym,amsmath,epsfig,amsthm} 
\usepackage{empheq}
\usepackage{here}
\usepackage{url}
\usepackage{ascmac}
\makeatletter
\usepackage{here}
\usepackage{comment}

\renewcommand\section{\@startsection {section}{1}{\z@}
{-30pt \@plus -1ex \@minus -.2ex}
{2.3ex \@plus.2ex}
{\normalfont\normalsize\bfseries\boldmath}}

\renewcommand\subsection{\@startsection{subsection}{2}{\z@}
{-3.25ex\@plus -1ex \@minus -.2ex}
{1.5ex \@plus .2ex}
{\normalfont\normalsize\bfseries\boldmath}}

\renewcommand{\@seccntformat}[1]{\csname the#1\endcsname. }

\makeatother
\newtheorem{theorem}{Theorem}
\newtheorem{lemma}{Lemma}

\theoremstyle{definition}
\newtheorem{definition}{Definition}

\begin{document}

\begin{center}
\uppercase{\bf Maximum Nim and Josephus Problem algorithm}
\vskip 20pt
{\bf Hikaru Manabe}\\
{\smallit Keimei Gakuin Junior and High School, Kobe City, Japan}\\
{\tt urakihebanam@gmail.com}
\vskip 10pt
{\bf Ryohei Miyadera }\\
{\smallit Keimei Gakuin Junior and High School, Kobe City, Japan}\\
{\tt runnerskg@gmail.com}
\vskip 10pt
{\bf Yuji Sasaki}\\
{\smallit Graduate School of Advanced Science and Engineering Hiroshima University, Higashi-Hiroshima City, Japan}\\
{\tt urakihebanam@gmail.com}
\vskip 10pt
{\bf Shoei Takahashi}\\
{\smallit Keio University Faculty of Environment and Information Studies, Fujisawa City, Japan}\\
{\tt shoei.takahashi@keio.jp}
\vskip 10pt
{\bf Yuki  Tokuni}\\
{\smallit Graduate School of Information Science University of Hyogo, Kobe City, Japan}\\
{\tt ad21o040@gsis.u-hyogo.ac.jp}
{\tt \ }


\end{center}

\centerline{\bf Abstract}
\noindent
Let $k,n$ be positive integers such that $n,k \geq 2$.
We present a new algorithm to determine the last number that remains in the Josephus problem, where numbers $1,2,\dots, n$ are arranged in a circle, and every $k$-th number is removed. 
We discovered a new formula that calculates the last number that remains in the Josephus 
problem, and this formula is based on a Maximum Nim of combinatorial game theory.

When $k \leq n$, the time complexity of our algorithm is $O(k \log n)$, which is on par with the existing $O(k \log n)$ algorithm. We do not have any recursion overhead or stack overflow because we do not use any recursion. Therefore, the space complexity of our algorithm is $O(1)$, and ours is better than the existing $O(k \log n)$ algorithm in this respect.
The algorithm is as follows.\\
$(i)$ We start with $x=k-1$. \\
$(ii)$ We substitute $x$ with $x +  \left\lfloor \frac{x}{k-1} \right\rfloor +1$ until we get $x$ such 
$nk-n \leq x$. \\
$(iii)$ Then, $nk-x$ is the number that remains.

\section{Introduction}
Let $\mathbb{Z}_{\ge 0}$ and $\mathbb{N}$ represent the set of non-negative numbers and the set of positive integers, respectively.
Let $n,k \in \mathbb{N}$.

We have a finite sequence of positive integers $1,2,3$
$, \cdots, n-1, n$ arranged in a circle, and we remove every $k$-th number until only one remains. The Josephus problem is to determine the number that remains.

Many researchers have tried to find an efficient algorithm to determine the number that remains in the Josephus problem.

D. Knuth \cite{knuth} proposed an $O(n \log n)$ algorithm, and E. Lloyd \cite{O(nlogm)} proposed an $O(n log k)$ algorithm for the case that $k < n$.
For a small $k$ and a large $n$, an $O(k \log n)$ algorithm is proposed in \cite{josephusalgorithm}.

We discovered a new formula that calculates the last number that remains in the Josephus 
problem, and this formula is based on the theory of a Maximum Nim of combinatorial game theory. We will present this formula in  \cite{integer2024} soon.
We made an algorithm using this formula. The time complexity of our algorithm for the Josephus problem is $O(k \log n)$. We do not have any recursion overhead or stack overflow because we do not use any recursion. Therefore, the space complexity of our algorithm is $O(1)$, and ours is better than the existing $O(k \log n)$ algorithm in this respect.

In this article, we omitted the mathematical background of the algorithm for the Josephus problem.
If you want to know our algorithm's mathematical background with proofs, please read our article \cite{integer2024}.

\begin{definition}\label{seconddefn}
For $k \in \mathbb{N}$ such that $k \geq 2$,
we define a function $h_{k}(x)$ for a non-negative integer $x$ as
\begin{equation}
h_{k}(x)= x + \left\lfloor \frac{x}{k-1} \right\rfloor +1. \nonumber  
\end{equation}    
\end{definition}

\begin{theorem}\label{coroforp}
Let $n$ be a natural number. Then, there exists $p \in \mathbb{N}$ such that 
$h_{k}^{p-1}(k-1) < n(k-1) \leq h_{k}^{p}(k-1)$, and the last number that remains is $nk-h_{k}^{p}(k-1)$ in the Josephus problem of $n$ numbers, where every $k$-th number is removed.
\end{theorem}
For a proof, see Corollary 1 of \cite{integer2024}.
\begin{lemma}\label{lemmaforvw}
For $n \in \mathbb{N}$, let
\begin{equation}
v=\frac{ \log n}{ \log k- \log(k-1)}\label{defofv}
\end{equation}
and
\begin{equation}
w=\frac{ \log(n+1)- \log2}{ \log k- \log(k-1)}.\label{defofw}
\end{equation}
Then, 
\begin{equation}
h_{k}^{\lfloor w \rfloor }(k-1)  \leq   n(k-1) \leq h_{k}^{\lceil v \rceil}(k-1). \label{hfloorhceil}
\end{equation}
\end{lemma}
For a proof, see Lemma 8 of \cite{integer2024}.

Based on Theorem \ref{coroforp},
we build the following algorithm for the Josephus problem. \\
Algorithm 1\\
$(i)$ We start with $x=k-1$. \\
$(ii)$ We substitute $x$ with $x +  \left\lfloor \frac{x}{k-1} \right\rfloor +1$ until we get $x$ such 
$nk-n \leq x$.\\
$(iii)$ Then, we get $m=nk-x$, and $m$ is the number that remains.

The following Theorem holds.
\begin{theorem}
For integers $n>0$ and $k>0$, The time complexity of Algorithm 1 is $O(k \log n)$.
\end{theorem}
\begin{proof}
It is sufficient to show that there exists $q \in \mathbb{N}$ such that 
\begin{equation}
n(k-1)\leq h^{q}_k(k-1)    
\end{equation}
and the order of $q$ is $O(k \log n)$.

Since
\begin{equation}
\lceil v \rceil=\frac{ \log n}{ \log k- \log(k-1)}+1 \leq k \log n +1 \label{defofv2}
\end{equation}
and
\begin{equation}
\lfloor w \rfloor=\frac{ \log(n+1)- \log2}{ \log k- \log(k-1)}-1 \geq (k-1)(\log(n+1)- \log2)-1.\label{defofw2}
\end{equation}
By (\ref{defofv2}) and (\ref{defofw2}), 
 Step $(ii)$ in Algorithm 1 is repeated in $O(k \log n)$ times. Step $(i)$ and $(ii)$ can be calculated in constant time. Therefore, the time complexity of Algorithm 1 is $O(k \log n)$.
\end{proof}

The time complexity of our algorithm is $O(k \log n)$, and this time complexity is on a par with a $O(k \log n)$ algorithm in \cite{josephusalgorithm}. We do not have any recursion overhead or stack overflow because we do not use any recursion. Therefore, the space complexity of our algorithm is $O(1)$, and ours is better than the existing $O(k \log n)$ algorithm in this respect.

The structure of our algorithm is more simple than 
the $O(k \log n)$ algorithm in \cite{josephusalgorithm}.

The following is the Python implementation of our algorithm. We have $n$ numbers 
$0,1,2,\dots, n-1$ and remove every $k$th number. We get the number that remains after the calculation.

\begin{verbatim}
def proposed_algorithm(n, k):
  x = k-1   # (1)
  range_min = (k - 1) * n 
  while x < range_min:   # (2)
    x += x // (k - 1) + 1   # (3)
  m = n * k - x   # (4)
  return  m - 1   # (5)
\end{verbatim}

\section{Comparison between Our Algorithm and Existing Algorithms}\label{sectioforalgo}
Here, we compare our algorithm to various existing algorithms. 

We present two existing algorithms and compare our algorithm to these.
\subsection{An Existing $O(n)$ Algorithm}\label{O(n)}
We have $n$ numbers 
$0,1,2,\dots, n-1$ and remove every $k$th number. We get the number that remains after the calculation.
In the following Python implementation, the algorithm is the same as the definition of the Josephus problem itself.
Here, we do not use recursion to avoid stack overflow.
\begin{verbatim}
def benchmark_b(n, k):
    if n == 1:
        return 0
    result = 0
    for i in range(2, n + 1):
        result = (result + k) % i
    return result
\end{verbatim}

$result = (result + k) \% i $
Since this algorithm does not use recursion, the space complexity is $O(1)$.Since time complexity is $O(n)$, the running time will increase rapidly as $n$ increases.

\subsection{An Existing $O(k \log n)$ Algorithm }\label{O(klogn)}
The following algorithm is based on the method of removing $k$-th, $2k$-th, ..., $\left\lfloor \frac{n}{k} \right\rfloor k$-
the number as one step. This algorithm used $cnt = n // k$ to remove these numbers as one step.

\begin{verbatim}
	def benchmark_b(n, k):
  if n== 1:
    return 0
  if k == 1:
    return n - 1
  if n < k:
    return (benchmark_b(n - 1, k) + k) % n
  
  cnt = n // k
  result = benchmark_b(n - cnt, k)
  result -= n % k
  if result < 0:
    result += n
  else:
    result += result // (k - 1)

  return result
\end{verbatim}
The time complexity of this algorithm is $O(k \log n)$. However, the space complexity is $O(n)$, and we may have 
the stack overflow because of recursion.
\subsection{Comparison of running time of each algorithm}
We compare our algorithm to two existing algorithms by program execution time.
We use a MacBook Pro(CPU: M1 Max, RAM:32GB) and run each program $3000$ times.

In the following statement, we denote the execution time of the algorithm in Subsection \ref{O(n)} by Benchmark A,
and the execution time of the algorithm in Subsection \ref{O(klogn)} by Benchmark B.

When $n$ is small, Benchmark A,B are small, but Benchmark A increases rapidly because of its time complexity $O(n)$ as $n$ increases.
When $n$ is large,  the execution time of our algorithm and Benchmark B are smaller than Benchmark A. However, if we compare the execution time of our algorithm to Benchmark B, ours is better.
The reason is that ours has a smaller space complexity, although both have the same time complexity $O(k\log n)$.
Therefore, when $n$ is large, our algorithm is better than these two algorithm.

Benchmark A is constant regardless of the size of $k$. The execution time of our algorithm and Benchmark B
increase as $k$ increases. However, Benchmark B does not increase endlessly because it does not depend on $k$ when $n<k$.

The efficiency of our algorithm, the algorithm in Subsection \ref{O(n)}, and the algorithm in Subsection \ref{O(klogn)} depend on $k$ and $n$.

When $k$ is small and $n$ is large, our algorithm and the algorithm in Subsection \ref{O(klogn)} are better than the algorithm in Subsection \ref{O(n)}. We can explain this fact by the time complexity $O(k \log n)$, but our algorithm is better than the algorithm in Subsection \ref{O(klogn)}.
When $k$ is large and $n$ is not large, the algorithm in Subsection \ref{O(n)} is better than others.




\end{document}